\documentclass[leqno,twoside,draft]{article}

\usepackage{geometry} 

\usepackage{graphicx,tikz} 
\usetikzlibrary{snakes,arrows,decorations.markings,intersections}

\tikzset{mygrid/.style={ultra thin}}    
\tikzset{myaxis/.style={thick,->,>=stealth}}    
\tikzset{mymain/.style={ultra thick}}   
\tikzset{inpt/.style={fill=black,semithick}}    
\tikzset{outpt/.style={fill=white,semithick}}    
\tikzset{prline/.style={semithick}}     
\tikzset{intline/.style={very thin}}         
\tikzset{asint/.style={thin}}            
\tikzset{myint/.style={ultra thick, decorate, decoration={coil,aspect=0,segment length=5pt,amplitude=.5pt}}}    

\tikzset{freccia/.style={>=stealth,decoration={markings,mark=at position .5 with {\arrow{stealth}}},postaction={decorate}}}
\tikzset{setline/.style={very thick}}

\tikzstyle tratt=[dotted]

\usepackage{float,floatflt} 

\usepackage{xcolor} 
\usepackage{booktabs}
\usepackage{array,multirow,colortbl} 
\newcolumntype{G}{>{\columncolor[gray]{0.9}}c}
\newcolumntype{H}{>{\columncolor[gray]{0.9}}l}
\definecolor{graymed}{gray}{0.9}

\usepackage[margin=20pt]{caption} 

\DeclareCaptionLabelSeparator{punto}{. }
\usepackage{microtype} 

\usepackage{enumitem}

\newlist{myenacomp}{enumerate}{10}
\setlist[myenacomp]{label=(\alph*),itemsep=.6ex,parsep=.3ex,topsep=1ex,partopsep=0ex}

\newlist{myen1comp}{enumerate}{10}
\setlist[myen1comp]{label=(\arabic*),itemsep=.6ex,parsep=.3ex,topsep=1ex,partopsep=0ex}

\newlist{myenicomp}{enumerate}{10}
\setlist[myenicomp]{label=(\roman*),itemsep=.6ex,parsep=.3ex,topsep=1ex,partopsep=0ex}

\newlist{myitbcomp}{itemize}{10}
\setlist[myitbcomp]{label=$\bullet$,itemsep=.6ex,parsep=.3ex,topsep=1ex,partopsep=0ex}

\newlist{mydcomp}{description}{10}
\setlist[mydcomp]{itemsep=.6ex,parsep=.3ex,topsep=1ex,partopsep=0ex}

\usepackage[latin1]{inputenc}
\usepackage[T1]{fontenc} 

\usepackage[english]{babel}

\usepackage{icomma} 
\usepackage{ellipsis} 

\usepackage{gensymb} 
\usepackage{eurosym}     

\usepackage{dingbat,bbding,harmony,phaistos} 

\DeclareMathAccent{\ring}{\mathalpha}{operators}{"17} 

\usepackage{amsmath} 
\usepackage{amssymb} 

\usepackage[ntheorem]{empheq} 

\usepackage[mathcal]{eucal} 
\usepackage{mathrsfs} 
\usepackage{upgreek} 

\usepackage{bbm,dsfont} 
\usepackage{mathbbol} 

\newcommand*{\ns}{\mathds}

\usepackage{braket} 
\newcommand*{\tc}{:}

\usepackage{mathtools} 

\DeclarePairedDelimiter{\abs}{\lvert}{\rvert}
\DeclarePairedDelimiter{\norm}{\lVert}{\rVert}
\DeclarePairedDelimiter{\uple}{\lparen}{\rparen}

\DeclarePairedDelimiter{\pg}{\{}{\}}
\DeclarePairedDelimiter{\pq}{\lbrack}{\rbrack}
\DeclarePairedDelimiter{\pt}{\lparen}{\rparen}
\DeclarePairedDelimiter{\pqt}{\lbrack}{\rparen}

\newcommand*{\sep}{,\,} 

\mathtoolsset{showmanualtags}

\newcommand*{\intaa}[3][]{\mathopen{#1 ]} #2 , #3 \mathclose{#1 [}} 

\newcommand*{\intca}[3][]{\mathopen{#1 [} #2 , #3 \mathclose{#1 [}}
\newcommand*{\intcc}[3][]{\mathopen{#1 [} #2 , #3 \mathclose{#1 ]}}

\newcommand*{\eqdef}{\vcentcolon=}

\newcommand{\eu}{\mathrm{e}}

\newcommand{\mpg}{\uppi}

\DeclareMathOperator{\nep}{\eu}

\DeclareMathOperator{\mpi}{\mpg}

\newcommand*{\clnf}{\colon} 

\newcommand*\dif{\mathop{}\!\mathrm{d}}
\newcommand*{\deriv}[3][]{\frac{\dif^{#1}#2}{\dif #3^{#1}}} 

\newcommand{\ve}{\boldsymbol} 
\newcommand*{\hor}{_{\mathrm h}} 
\newcommand*{\vrt}{_{3}} 
\newcommand*{\hf}[1]{\overline{#1}^{\mathrm h}} 
\newcommand*{\cs}{\mathrm} 
\newcommand*{\fs}{\mathscr} 
\newcommand*{\loc}{_{\mathrm{loc}}} 
\newcommand*{\trasl}[1]{_{+t}}

\newcommand*{\abset}{\mathcal B}
\newcommand*{\attr}{\mathcal A}

\usepackage{framed,mdframed}
\usepackage[framed,amsmath,thmmarks]{ntheorem} 

\theoremstyle{change}
\theoremheaderfont{\normalfont\bfseries\sffamily}
\theorembodyfont{\normalfont}
\theoremseparator{}
\theoremnumbering{arabic} 
\theoremsymbol{}

\newtheorem{theorem}{Theorem}[section]
\newtheorem{proposition}[theorem]{Proposition}
\newtheorem{lemma}[theorem]{Lemma}
\newtheorem{remark}[theorem]{Remark}

\theoremstyle{changebreak}

\theoremstyle{change}
\theorembodyfont{\normalfont\sffamily}

\theorembodyfont{\normalfont\slshape}

\newtheorem{definition}[theorem]{Definition}

\newcommand{\textdef}[1]{\textbf{#1}}

\theoremstyle{nonumberplain}
\theoremheaderfont{\normalfont\scshape}
\theorembodyfont{\normalfont}
\theoremseparator{.}
\theoremsymbol{\ensuremath{\blacksquare}}
\newtheorem{proof}{Proof}

\usepackage{verbatim}

\definecolor{refkey}{rgb}{1,1,1}
\definecolor{labelkey}{rgb}{0,0,1}
\definecolor{refkey}{rgb}{1,1,1}

\newcommand\D{\partial}
\newcommand{\bw}{\boldsymbol{w}}
\newcommand{\bbf}{\boldsymbol{f}}

\begin{document}

\title{Global Attractor for the Navier--Stokes Equations with
  horizontal filtering} \author{Luca Bisconti and Davide Catania}

\maketitle

\subsection*{Abstract}

We consider a Large Eddy Simulation model for a homogeneous
incompressible Newtonian fluid in a box space domain with periodic
boundary conditions on the lateral boundaries and homogeneous
Dirichlet conditions on the top and bottom boundaries, thus simulating
a horizontal channel. The model is obtained through the application of
an anisotropic horizontal filter, which is known to be less memory
consuming from a numerical point of view, but provides less regularity
with respect to the standard isotropic one defined as the inverse of
the Helmholtz operator.

It is known that there exists a unique regular weak solution to this
model that depends weakly continuously on the initial datum. We show
the existence of the global attractor for the semiflow given by the
time-shift in the space of paths. We prove the continuity of the
horizontal components of the flow under periodicity in all directions
and discuss the possibility to introduce a solution semiflow.

\medskip

\textbf{Keywords.} Navier--Stokes equations, anisotropic filters,
global attractor, time-shift semiflow, Large Eddy Simulation (LES),
turbulent flows in domains with boundary, approximate deconvolution
methods (ADM).

\textbf{Subjects.} 76D05; 35Q30, 76F65, 76D03.

\section{Introduction}
Incompressible fluids with constant density are described by the
Navier--Stokes equations
\begin{align}
  & \label{eq:NS-3D-a}
  \partial_t \ve u+\nabla\cdot\pt{\ve u\otimes \ve u}-
  \nu\Delta \ve u+\nabla\mpi=\ve f \, ,\\
  & \label{eq:NS-3D-b} \nabla\cdot \ve u = 0 \, ,
\end{align}
supplemented with initial and boundary conditions, where $\ve u(t, \ve
x)=(u_1, u_2, u_3)$ is the velocity field, $\pi(t, \ve x)$ denotes the
pressure, $\ve f (t, x) = (f_1, f_2, f_3)$ is the external force, and
$\nu > 0$ the kinematic viscosity.

In order to perform numerical simulations of the $3$-dimensional fluid
equations \eqref{eq:NS-3D-a}--\eqref{eq:NS-3D-b}, many regularization
methods have been proposed.  Among these models, in the recent years,
there has been a lot of interest, from both the point of view of the
pure and applied mathematicians, around the so called
``$\alpha$-models''.  These models are based on a filtering obtained
through the application of the inverse of the Helmholtz operator
\begin{equation} \label{eq:Helm} A=I-\alpha^2 \Delta\, ,
\end{equation}
where $\alpha >0$ is interpreted as a spatial filtering scale.

In this paper, we are concerned with a regularized model for the $3$D
Navier--Stokes equations derived by the introduction of a suitable
horizontal (anisotropic) differential filter and we prove the
existence of a global attractor for the corresponding time-shift
dynamical system in path-space.  Let us consider
\begin{gather*}
  \ve x = \uple{x_1 \sep x_2 \sep x_3}\, , \qquad \ve x\hor = \uple{x_1\sep x_2}\, , \\
  \partial_j=\partial_{x_j}\, , \qquad \Delta\hor
  = \partial_1^2+\partial_2^2\, , \qquad
  \nabla\hor=\uple{\partial_1\sep\partial_2}\, ,
\end{gather*}
where ``h'' stays for ``horizontal'' and, instead of choosing the
filter given by \eqref{eq:Helm}, we take into account the horizontal
filter given by (see \cite{BER:2012})
\begin{equation} \label{eq:Helm-hor} A\hor=I-\alpha^2 \Delta\hor
  \qquad \textrm{ with } \alpha>0\, .
\end{equation}
As discussed in \cite{ALI:2013, GERM:1986, LAYT:2003}, from the point
of view of the numerical simulations, this filter is less memory
consuming with respect to the standard one.  Another significant
advantage of this choice is that there is no need to introduce
artificial boundary conditions for the Helmholtz operator.

The idea behind anisotropic differential filters can be traced back to
the approach used by Germano \cite{GERM:1986}. Recently, the Large
Eddy Simulation (LES) community has manifested interest in models
involving such a kind of filtering (e.g, \cite{ALI:2013,
  BE-IL-LAY:2006, DEA:2006, SCO:1993}) and the connection with the
family of $\alpha$-models has been highlighted and investigated by
Berselli in \cite{BER:2012}. Exploiting the smoothing provided by the
horizontal filtering \eqref{eq:Helm-hor}, Berselli has proved global
existence and uniqueness of a proper class of weak solutions to the
considered regularized model (see the system of equations
\eqref{third.1}--\eqref{third.2} below). Again, motivated by
\cite{BER:2012}, the authors of \cite{BER-CAT:2013,BER-CAT:2014} have given a
considerable mathematical support to the well-posedness of
initial-boundary value problems, in suitable anisotropic Sobolev
spaces, to the $3$D~Boussinesq equations with horizontal filter for
turbulent flows.

In the sequel, we mainly consider the domain
\begin{align*}
  D = \set{\ve x \in \ns R^3 \tc -\mpi L<x_1, x_2<\mpi L, -d<x_3<d}\,
  ,
\end{align*}
$L>0$, with $2\mpi L$ periodicity with respect to $\ve x\hor$
(i.e. with respect to $x_1, x_2$), and homogeneous Dirichlet boundary
conditions on $\mathit\Gamma \eqdef  \set{\ve x\in\ns R^3
  \tc -\mpi L<x_1, x_2<\mpi L,\, x_3=\pm d}$.  Observe that the filter
given by \eqref{eq:Helm-hor} is acting just on the horizontal
variables, so it makes sense to require the periodicity only in $\ve
x\hor = \uple{x_1 \sep x_2}$.

Set $\ve w = \hf{\ve u}= A\hor^{-1}\ve u$ and
$q=\hf{\mpi}=A\hor^{-1}\mpi$, so that $\ve u = A\hor\ve w$. Then,
applying the horizontal filter ``$\overline{{\,\,\,\,}^{{}^{}}}^{\text{h}}$''
component by component to the various fields and tensor fields in
\eqref{eq:NS-3D-a}--\eqref{eq:NS-3D-b}, and solving the interior
closure problem by the approximation
\begin{equation} \label{eq:closure} \hf{\ve u\otimes\ve u}\approx
  \hf{\hf{\ve u}\otimes\hf{\ve u}}=\hf{\ve w\otimes\ve w}\, ,
\end{equation}
we finally get the regularized model
\begin{align}
  & \label{third.1}
  \partial_t \ve w+\nabla\cdot\hf{\pt{\ve w\otimes \ve w}}-
  \nu\Delta \ve w+\nabla q=\hf{\ve f} \, ,\\
  & \label{third.2} \nabla\cdot \ve w = 0 \, .
\end{align}
Here, we assume homogeneous Dirichlet boundary conditions on
$\mathit\Gamma $ for the filtered fields as well as for the
unfiltered ones, in order to prevent the introduction of artificial
boundary conditions, and impose the initial datum $\ve w|_{t=0} =\ve
w_0$ for the filtered velocity field $\ve w$.

Let us note that this model represents a special case of Approximate
Deconvolution LES Model (ADM), see Adams--Stolz \cite{ADA-STO:2001},
when the order of deconvolution is zero. We refer to
\cite{BER:2012,ALI:2013,BIS:2013,BE-CA-LE:2013} for some recent
results in this context concerning general orders of deconvolution.

Our aim is to prove the existence of an attractor in the class of
\emph{regular weak solutions} (see below for details) to the
horizontally filtered model \eqref{third.1}--\eqref{third.2}. However,
the present case does not seem to fit the classical theory of
attractors (see, e.g., \cite{BA-VI:1992}) and a different scheme is
needed to carry out our analysis.  In fact, despite the smoothing
created by the horizontal filter, the regularity of the considered
weak solutions does not ensure the continuous dependence on their
initial data, even in the fully periodic setting (the dependence on
the initial data is only weakly continuous). Hence, the standard
dynamical theory fails to apply to this situation since the strong
continuity on the initial data is needed to get the continuity of the
solution semiflow.

To overcome this problem, we follow the approach proposed by Sell
\cite{Sell:JDDE}: in this case, the dynamics becomes the time-shift in
the space of paths, and the attractor is a suitable compact set that
attracts the regular weak solutions under the action of the time-shift
$S(t)\ve w(\cdot)=\ve w(t+\cdot)$.

Let $\mathcal W$ be the space of the regular weak solutions to
\eqref{third.1}--\eqref{third.2}, denote by $H$ and $V$ the usual
function spaces of fluid dynamics (see Subsection~\ref{subsec.fs}),
and set
\begin{align*}
  V\hor &=\set{\ve\phi \in H \tc \nabla\hor
    \ve\phi \in \cs L^2(D)}\, , \\
  \cs H\hor^2 &= \set{\ve \phi \in V\hor \tc \nabla\hor\nabla\phi \in
    \cs L^2(D)}\, .
\end{align*}
We prove the following result (see Section~\ref{sec.attractor}).

\begin{theorem} \label{th.tsattr} Given $\ve w_0 \in V\hor$ and $\ve f
  = \ve f(\ve x)\in (V\cap \cs H^2\hor)^*$, the time-shift $S(t)$ in
  $\mathcal W$ has a unique global attractor.
\end{theorem}
Here, for the sake of simplicity, we assume that the forcing term $\ve
f$ is independent of time.

\medskip

In the last part of the paper, we consider problem
\eqref{third.1}--\eqref{third.2} in the fully space-periodic
setting. Notice that, following analogous computations,
Theorem~\ref{th.tsattr} can be readily adapted to the simplified
situation given by space-periodic boundary conditions in all
directions.

We want to discuss the possibility to obtain an analogous result for
the semigroup $\pt[\big]{S(t)}_{t\geq 0}$ on $V\hor$, where
$S(t)\clnf V\hor\rightarrow V\hor$ is given by $S(t)\ve w_0=\ve
w(t,\cdot)$, $\ve w_0\in V\hor$, and $\ve w(t,\cdot)$ is the solution
at time $t$ to \eqref{third.1}--\eqref{third.2} corresponding to the
initial datum $\ve w_0$.
In such a case, the main issue is to prove that $S(t)$ effectively
defines a semiflow (the \emph{solution semiflow}) and, in particular,
that $S(t)$ is a continuous operator.  In fact, the regular weak
solutions to \eqref{third.1}--\eqref{third.2} depend just in a weakly
continuously way on their initial data, giving no guarantees on the
continuity of $S(t)$.

In Section~\ref{sec.periodic}, we show that the horizontal components
of the flow, i.e. $\ve w\hor$, depend continuously on the initial
datum, proving the following.

\begin{proposition} \label{prop.wh} If $\ve w$ is a regular weak
  solution to \eqref{third.1}--\eqref{third.2} under the periodic
  setting, then $\ve w\hor \in \fs C^0\pt{\intca{0}{T};V\hor}$ for
  each $T>0$.
\end{proposition}

\smallskip

Observe that the continuity of the solution operator $S(t)$ could be
shown by proving that $\partial_t \nabla\hor\ve w \in \cs L^2
V^*$; nevertheless, we only have that $\partial_t \nabla\hor\ve w\hor \in
\cs L^2 V^*$.
However, since the considered problem admits a unique regular weak
solution, we may wonder whether it is possible to get more regularity
for such a solution, also in the vertical component, by exploiting
again the special features of the fully periodic space-domain.  To
this end, we give an improved regularity result (see
Theorem~\ref{thm:estimate-H2_h} below) which actually shows that,
although the regularization created by the filter is strong in the
horizontal components (and in the derivatives with respect to the
horizontal components), this smoothing is not so effectively
transported to the vertical component, even in the space-periodic
case.  Though Theorem~\ref{thm:estimate-H2_h} is not directly useful
to prove the continuity of $S(t)$ (in fact, it might be more
appropriate in order to get compactness properties of $S(t)$), it
seems interesting by itself and we report it, at the end of
Section~\ref{sec.periodic}, for the reader's convenience.

Because of all the above facts, different techniques seem necessary to
get the continuity of the solution operator $S(t)$ associated to
\eqref{third.1}--\eqref{third.2}, and to prove the existence of the
global attractor in such a case.  We will address these issues in a
future paper, in which we will study more thoroughly the dynamics
associated to problem \eqref{third.1}--\eqref{third.2}, possibly
supplemented with homogeneous Dirichlet boundary conditions.

\subsection{Plan of the paper} In Section \ref{sec.setting} we
describe the functional setting, the horizontal filter and the notion
of regular weak solution, and we recall a known result
concerning the existence and uniqueness of such solutions.
In Section~\ref{sec.results}, we recall the notion of global attractor and
describe the main results. Section~\ref{sec.estimates} is devoted to
the proof of some basic estimates that will be used subsequently. In
Section~\ref{sec.attractor}, under horizontal periodicity and
homogeneous Dirichlet boundary conditions on $\mathit\Gamma$, the
existence of the global attractor defined through the shifting
semiflow is given. Finally, in Section~\ref{sec.periodic}, when the domain
is periodic in all directions, we prove the
continuity of the horizontal components of the flow,  and we also provide an improved
regularity result for the solutions to \eqref{third.1}--\eqref{third.2}, i.e.
Theorem~\ref{thm:estimate-H2_h}.

\section{Functional setting and anisotropic
  filtering} \label{sec.setting}

\subsection{Function spaces} \label{subsec.fs}

We introduce the following function spaces:
\begin{align*}
  & \cs L^2(D) \eqdef \set{\phi\clnf D \rightarrow \ns R \text{
      measurable, } 2\pi L \text{ periodic in } \ve x\hor, \, \int_D
    \abs{\phi}^2 \dif \ve x < +\infty},
  \\
  & \cs L^2_0(D) \eqdef \set{\phi \in \cs L^2(D) \text{ with zero mean
      with respect to } \ve x\hor} ,
  \\
  & H \eqdef \set{\ve\phi \in (\cs L^2_0(D))^{3} \tc
    \nabla\cdot\ve\phi=0 \text{ in } D, \, \ve\phi\cdot \ve n = 0
    \text{ on } \mathit\Gamma},
\end{align*}
($\ve n$ is the outward normal to $\mathit\Gamma$), all with $\cs L^2
$ norm denoted by $\norm{\cdot}$, and scalar product
$\uple{\cdot\sep\cdot}$ in $\cs L^2$. Here and in the following, we
use the same notation for scalar and vector valued functions, when
they make sense. Moreover, we set
\begin{align*}
  & V\hor \eqdef \set{\ve\phi \in H \tc \nabla\hor \ve\phi \in \cs
    L^2(D)}\, ,
  \\
  & V \eqdef \set{\ve\phi \in H \tc \nabla\ve\phi\in \cs L^2(D) \text{
      and } \ve\phi=\ve 0 \text{ on } \mathit\Gamma\,}\, ,
  \\
  & \cs H\hor^2 \eqdef \set{\phi \in \cs H^1\hor \tc
    \nabla\hor\nabla\phi \in \cs L^2(D)} \, ,
\end{align*}
where the definitions of $V\hor$ and $\cs H^2\hor$ have been recalled
for the reader's convenience, and denote by $V^*$ the topological dual
space to $V$. We denote by $\cs L^p$ and $\cs H^m$ classical Lebesgue
and Sobolev spaces. Continuous and weakly continuous functions are
denoted respectively by the symbols $\fs C$ and $\fs C_{\rm w}$.

Given a Banach space $X$ with norm $\norm{\cdot}_X$ and $p
\in \intca{1}{+\infty}$, we denote  by $\cs L^p\loc\pt*{0,\infty;X}$ the usual Bochner space
formed by  functions $\phi\clnf\intaa{0}{\infty}\rightarrow X$ such that, for all $0<a\leq
b<+\infty$, the $\cs L^p\pt{a,b}$ norm of $\norm{\phi(\cdot)}_X$ is
finite. We will also denote by $\cs L^p\loc\pqt*{0,\infty;X}$ the
space of functions in $\cs L^p\loc\pt*{0,\infty;X}$ such that the $\cs
L^p\pt{0, b}$-norm of $\norm{\phi(\cdot)}_X$ is finite for every
$b\in\intaa{0}{+\infty}$, and analogously for
$\cs H^m\loc\pqt*{0,\infty;X}$ (see, e.g., \cite{SELL-YOU:2002}).

\subsection{Basic results for the filtered
  model} \label{subsec.filtered} Applying the horizontal filter given
by \eqref{eq:Helm-hor} to the various terms in the Navier--Stokes
equations \eqref{eq:NS-3D-a}--\eqref{eq:NS-3D-b}, and modifying the
nonlinear quadratic term by means of the approximation
\eqref{eq:closure} (interior closure problem in LES), we have deduced
 the approximate model \eqref{third.1}--\eqref{third.2}, i.e.
\begin{align*}
  &   \partial_t \ve w+\nabla\cdot\hf{\pt{\ve w\otimes \ve w}}-
  \nu\Delta \ve w+\nabla q=\hf{\ve f} \, ,\\
  &   \nabla\cdot \ve w = 0 \, .
\end{align*}

Let us recall the following definition.
\begin{definition} \label{def:reg-weak-sol} We say that $\ve w \clnf
  \intca{0}{T}\times D \rightarrow \ns R^3 $ is a \textdef{regular
    weak solution} (omitting the pressure term $q$) to
  \eqref{third.1}--\eqref{third.2}, with $\ve f 
  \in (V\cap \cs H^2\hor)^*$ independent of time (for simplicity), and
  ${\ve w(0)} = {\ve w_0}\in V\hor$ in weak sense, when the following
  properties are verified.
  \begin{itemize}
  \item Regularity:
    \begin{align*}
      & \ve w\in \cs L^\infty\pt{0,T;V\hor}\cap \cs L^2\pt{0,T;V\cap \cs H^2\hor}
      \cap \fs C_{\mathrm w}\pt{0,T;V\hor} \, , \\
      \    & w\vrt \in \cs L^\infty\pt{0,T;\cs H^1}\cap \cs L^2\pt{0,T;\cs H^2} \, ,\\
      & \partial_t \ve w \in \cs L^2\pt{0,T;V^*} \, .
    \end{align*}
  \item Weak formulation:
    \begin{align*}
      & \int_0^{+\infty} \pg*{\uple{\ve w\sep\partial_t
          \ve\varphi}-\nu\uple{\nabla\ve
          w\sep\nabla\ve\varphi}+\uple{\ve w\otimes\ve w\sep\nabla
          \hf{\ve\varphi}}}(s)\dif s = -\int_0^{+\infty} \uple{\ve
        f\sep\ve\varphi}(s)\dif s -\uple[\big]{\ve
        w(0)\sep\ve\varphi(0)}
    \end{align*}
    for each $\ve\varphi\in\pt[\big]{\fs
      C_0^\infty(D\times\intca{0}{T})}^3$ such that
    $\nabla\cdot\ve\varphi=0$ or, equivalently (up to a modification
    of $\ve w$ on a set of zero measure, see Galdi~\cite{Galdi:2000}),
    \begin{align*}
      & \int_{t_0}^{t_1} \pg*{\uple{\ve w\sep\partial_t \ve\varphi}-\nu
        \uple{\nabla\ve w\sep\nabla\ve\varphi}+
        \uple{\ve w\otimes\ve w\sep\nabla \hf{\ve\varphi}}}(s)\dif s \\
      & \qquad = -\int_{t_0}^{t_1} \uple{\ve f\sep\ve\varphi}(s)\dif s
      + \uple[\big]{\ve w(t_1)\sep\ve\varphi(t_1)} - \uple[\big]{\ve
        w(t_0)\sep\ve\varphi(t_0)}
    \end{align*}
    for each $0\leq t_0 \leq t_1 < T$.
  \end{itemize}
\end{definition}

\begin{remark} \label{rem:cont} 
  Since $\ve w\in \cs
  L^2\pt{0,T;V}$ and $\partial_t \ve w \in \cs L^2\pt{0,T;V^*}$, we
  deduce by classical interpolation results that $\ve w\in \fs
  C\pt{\intca{0}{T};H}$ (see~\cite{Lio-Mag:68,SELL-YOU:2002}).
\end{remark}

We state the following existence theorem.

\begin{theorem}
  Let be given $\ve w_0\in V\hor$, $\ve f=\ve f(\ve x)\in (V\cap \cs
  H^2\hor)^*$ and $\nu> 0$. Then there exists a unique regular weak
  solution to \eqref{third.1}--\eqref{third.2}, with ${\ve w(0)} =
  {\ve w_0}$, depending in a weakly continuous way on the
  data. Moreover, the solution satisfies the energy (of the model)
  identity
  \begin{equation} \label{eq:estimate-existence}
    \begin{aligned}
      &\frac{1}{2}\pt*{\norm{\ve w(t)}^2+\alpha^2\norm{\nabla\hor\ve w(t)}^2} +\nu\int_0^t\pt*{\norm{\nabla\ve w(s)}^2+\alpha^2\norm{\nabla\hor\nabla\ve w(s)}^2}\!\dif s \\
      & \qquad = \frac{1}{2}\pt*{\norm{\ve
          w(0)}^2+\alpha^2\norm{\nabla\hor\ve
          w(0)}^2}-\int_0^t\uple[\big]{\ve f\sep A\hor \ve w(s)}\dif
      s\, .
    \end{aligned}
  \end{equation}
\end{theorem}

This result can be obtained as in \cite{BER:2012,BER-CAT:2013}, where
a space-domain $\intaa{-\mpi}{\mpi}^2\times\intaa{-d}{d}$ periodic in
$(x_1, x_2)$ and with homogeneous Dirichlet boundary conditions on
$x_3=\pm d$ is considered.

\smallskip

Also, notice that the regularity of $\ve f$ implies $\hf{\ve f}\in \cs
L^2_0(D)$.

\section{Attractors and main results} \label{sec.results}

Let $(W,d)$ be a metric space. A \textdef{semigroup} on $(W,d)$ is a
family of operators $\pt[\big]{S(t)}_{t\geq 0}$, $S(t)\clnf
W\rightarrow W$, that satisfies $S(0)w=w$ and $S(s)S(t)w=S(t+s)w$ for
each $w\in W$ and for every $s,t\geq 0$.

A \textdef{semiflow} on $(W,d)$ is a mapping $\sigma\clnf
\intca{0}{\infty}\times W \rightarrow W$ defined by
$\sigma(t,w)=S(t)w$, where $\pt[\big]{S(t)}_{t\geq 0}$ is a semigroup,
and such that the restriction $\sigma\clnf \intaa{0}{\infty} \times W
\rightarrow W$ is continuous.

We say that $\attr\subset W$ is a \textdef{global attractor} for the
semiflow if $\attr$ is nonempty and compact, $S(t)\attr=\attr$ for all
$t\geq 0$ (i.e. $\attr$ is invariant), and for all bounded sets
$B\subset W$, we have $\lim_{t\rightarrow +\infty}
\delta\pt[\big]{S(t)B\sep\attr}=0$, where $\delta(X\sep Y)\eqdef
\sup_{x\in X} \inf_{y\in Y} d(x, y)$ is the Hausdorff semidistance
between the pair of sets $X,Y\subset W$. A global attractor is
necessarily unique (and it coincides with the omega-limit of an
absorbing set, see Section \ref{sec.attractor}).

\smallskip

Let us note that, for each $T>0$, a regular weak solution to
\eqref{third.1}--\eqref{third.2} in $\intca{0}{T}\times D$ is defined,
so that, by uniqueness, we get a unique solution $\ve w$ defined in
$\intca{0}{\infty}\times D$, and, in particular, $\ve w\in \cs
L^\infty\loc\pqt{0,\infty;V\hor}\cap \cs L^2\loc\pqt{0,\infty;V\cap
  \cs H^2\hor}$. Indeed, we can prove that $\ve w\in \cs
L^\infty\pt{0,\infty;V\hor}\cap \cs L^2\loc\pqt{0,\infty;V\cap \cs
  H^2\hor}$, as shown in Theorem \ref{th.estimates}.

Thus we can consider as $W$ the set $\mathcal W$ of regular weak
solutions to \eqref{third.1}--\eqref{third.2} in $\cs
L^\infty\pt*{0,\infty;V\hor}\cap \cs L^2\loc\pqt{0,\infty;V\cap \cs
  H^2\hor}$, with the metric $d$ induced by $\cs
L^2\loc\pqt{0,\infty;V\hor}$:
\begin{gather*}
  d(\ve w_1, \ve w_2) = \sum_{n=1}^\infty 2^{-n} \min\pg*{ 1 \sep
    \norm{\ve w_1 - \ve w_2}_{\cs L^2\pt*{0,n;V\hor} }}\, .
\end{gather*}


  Notice that $\mathcal W$ is closed in $\cs
L^2\loc\pqt{0,\infty;V\hor}$, which is complete with respect to $d$,
thus $(\mathcal W,d)$ is a complete metric space; this is one of the consequences
of the proof of Proposition~\ref{prop.comp} below.

We use the time-shift operator to define the semigroup and hence the
semiflow: $S(t)\ve w=\ve w\trasl{t} \eqdef \ve w(\cdot+t)$, for each
$\ve w\in \mathcal W$. The existence of the global attractor for the
time-shift $S(t)$, i.e. Theorem~\ref{th.tsattr}, is proved in
Section~\ref{sec.attractor}.

%

%
%
%

%
%
%
%


\section{Preliminary Estimates} \label{sec.estimates}


We set $\lambda_1>0 $ equal to the first eigenvalue of the Stokes
operator with horizontal periodic conditions and homogeneous Dirichlet
boundary conditions on $\mathit\Gamma$, projected on the space of
divergence free functions. We recall that $\lambda_1$ can be used as
the constant in the Poincar\'e inequality. Moreover, we set
$\Lambda\hor = (-\Delta\hor)^{1/2}$ with domain $\cs H^2\hor$ and the
same mixed periodic-Dirichlet conditions as above ($\Lambda\hor$ is a
positive self-adjoint operator), and
\begin{equation} \label{eq:utilities-1}
  \begin{aligned}
    & k_0(t) = \norm{\ve w(t)}^2 + \alpha^2 \norm{\nabla\hor\ve w(t)}^2 \, ,\\
    & K_1 = \min \pg*{\norm{\Lambda\hor^{-1}\ve f}^2 \sep
      \norm{\Lambda\hor^{-1/2}\ve f}^2} \, , \\
    & k_1(t) = k_0(t) + \dfrac{K_1}{\nu\lambda_1}
  \end{aligned}
\end{equation}
for each $t\geq 0$.

\begin{theorem} \label{th.estimates} If $\ve w$ is the regular weak
  solution to \eqref{third.1}--\eqref{third.2} in the time interval
  $\intca{0}{\infty}$, $t\geq 0$ and $r>0$,
  then we have
  \begin{align}
    & \norm{\ve w(t+r)}^2+\alpha^2\norm{\nabla\hor\ve w(t+r)}^2 \leq
    k_0(t)\nep^{-\nu\lambda_1 (t+r)}
    + \dfrac{K_1}{\nu\lambda_1}\pt{1-\nep^{-\nu\lambda_1 (t+r)}} \leq k_1(t)\, , \label{est.LinfVh} \\
    & \int_t^{t+r} \pt*{ \nu\norm{\nabla\ve
        w(s)}^2+\nu\alpha^2\norm{\nabla\hor\nabla\ve w(s)}^2 } \dif s
    \leq rK_1 + k_1(t)\, . \label{est.espr} 
  \end{align}
  In particular, we have $\ve w \in \cs
  L^\infty\pt*{0,\infty;V\hor}\cap \cs L^2\loc\pqt{0,\infty;V\cap \cs
    H^2\hor}$.
\end{theorem}

\begin{proof}
  Let us consider the Galerkin approximate solutions
  \begin{align*}
    &\ve w_m(t,\ve x) = \sum_{j=1}^m g_m^j(t) \ve E_j(\ve x) \, ,
  \end{align*}
  where $\ve E_j$ are smooth eigenfunctions of the Stokes operator on
  $D$, with periodicity in $\ve x$. If $\ve P_m$ denotes the
  projection on $\text{span}\set{\ve E_1,\ldots,\ve E_m} $, then ${\ve
    w_m}$ solves the Cauchy problem
  \begin{align*}
    &\deriv{}{t} \uple{\ve w_m\sep\ve E_i} +\nu\uple{\nabla\ve
      w_m\sep\nabla\ve E_i}-
    \uple{\ve w_m\otimes\ve w_m\sep\nabla\hf {\ve E}_i} = \uple{\ve f\sep\ve E_i} \, ,\\
    &\ve w_m(0)= \ve P_m \pt[\big]{\ve w(0)}\, ,
  \end{align*}
  for $i=1,\ldots, m$. As shown in \cite{BER:2012}, we have that, up
  to considering a subsequence,
  \begin{align*}
    \ve w_m \rightarrow \ve w \qquad \text{in} \quad \cs
    L^2\pt{0,T;\cs L^2(D)}
  \end{align*}
  as $m\rightarrow +\infty$, where $\ve w_m$ is the regular weak
  solution to \eqref{third.1}--\eqref{third.2}, and this is sufficient
  to pass to the limit in the nonlinear term.

  \medskip

  We test the first equation against $A\hor\ve w_m$ and use (see
  \cite{BER:2012})
  \begin{align*}
    -\uple{\ve w_m\otimes\ve w_m\sep \nabla\hf{A\hor\ve w_m}} =
    \uple{\hf{\nabla\cdot\uple{\ve w_m\otimes\ve w_m}}\sep A\hor\ve
      w_m} = 0
  \end{align*}
  to get
  \begin{align*}
    \begin{split}&\frac{1}{2}\deriv{}{t}\pt*{\norm{\ve
          w_m}^2+\alpha^2\norm{\nabla\hor\ve
          w_m}^2}+\nu\pt*{\norm{\nabla\ve
          w_m}^2+\alpha^2\norm{\nabla\hor\nabla\ve w_m}^2} =
      \uple{\hf{\ve f}\sep A\hor\ve w_m} = \uple{{\ve f}\sep \ve w_m}
      \, .
    \end{split}\end{align*}

  We estimate the righthand side by
  \begin{align*}
    \abs{\uple{{\ve f}\sep \ve w_m}} &\leq \begin{cases}
      &\norm{\Lambda\hor^{-2}\ve f} \, \norm{\Delta\hor\ve w_m} \\*[1em]
      &\norm{\Lambda\hor^{-1}\ve f} \, \norm{\Lambda\hor\ve w_m}
    \end{cases}
    \leq \begin{cases} &\dfrac{\norm{\Lambda\hor^{-2}\ve
          f}^2}{2\nu\alpha^2} +
      \dfrac{\nu\alpha^2\norm{\Delta\hor\ve w_m}^2}{2} \\*[1em]
      &\dfrac{\norm{\Lambda\hor^{-1}\ve f}^2}{2\nu} +
      \dfrac{\nu\norm{\Lambda\hor \ve w_m}^2}{2}
    \end{cases} \\
    & \leq \begin{cases} &\dfrac{\norm{\Lambda\hor^{-2}\ve
          f}^2}{2\nu\alpha^2} +
      \dfrac{\nu\alpha^2\norm{\nabla\hor\nabla\ve w_m}^2}{2} \\*[1em]
      &\dfrac{\norm{\Lambda\hor^{-1}\ve f}^2}{2\nu} +
      \dfrac{\nu\norm{\nabla\hor\ve w_m}^2}{2}
    \end{cases} \\
    & \leq \dfrac{1}{2}\pt*{K_1+ \nu\norm{\nabla\hor\ve w_m}^2 +
      \nu\alpha^2\norm{\nabla\hor\nabla\ve w_m}^2}
  \end{align*}
  by Cauchy--Schwarz inequality and definition of $K_1$. We deduce
  \begin{align} \label{est.K11}
    \begin{split}&\deriv{}{t}\pt*{\norm{\ve
          w_m}^2+\alpha^2\norm{\nabla\hor\ve
          w_m}^2}+\nu\pt*{\norm{\nabla\ve
          w_m}^2+\alpha^2\norm{\nabla\hor\nabla\ve w_m}^2} \leq K_1
    \end{split}\end{align}
  and, thanks to the Poincar\'e inequality,
  \begin{align*}
    \begin{split}&\deriv{}{t}\pt*{\norm{\ve
          w_m}^2+\alpha^2\norm{\nabla\hor\ve w_m}^2}+\nu\lambda_1
      \pt*{\norm{\ve w_m}^2+\alpha^2\norm{\nabla\hor\ve w_m}^2} \leq
      K_1\, ;
    \end{split}\end{align*}
  an application of Gronwall's lemma in the interval $\intcc{t}{t+r}$
  yields
  \begin{align*}
    &\norm{\ve w_m(t+r)}^2+\alpha^2\norm{\nabla\hor\ve w_m(t+r)}^2 \\
    & \qquad \leq \pt*{\norm{\ve w_m(t)}^2 + \alpha^2
      \norm{\nabla\hor\ve w_m(t)}^2}\nep^{-\nu\lambda_1 (t+r)} +
    \dfrac{K_1}{\nu\lambda_1}\pt{1-\nep^{-\nu\lambda_1 (t+r)}}
  \end{align*}
  and, taking the limit as $m\rightarrow +\infty$, we obtain
  \eqref{est.LinfVh} and in in particular
  \begin{align*}
    \norm{\ve w(t+r)}^2+\alpha^2\norm{\nabla\hor\ve w(t+r)}^2 \leq
    k_1(t) \, ,
  \end{align*}
  which implies
  \begin{align*}
    \ve w \in \cs L^\infty\pt*{0,\infty;V\hor}\,
  \end{align*}
  by taking $t=0$ and using the fact that $k_1(0)$ is finite (and
  independent of time), since $\ve w_0\in V\hor$.

  Integrating \eqref{est.K11} over $s\in\intcc{t}{t+r}$, taking
  $m\rightarrow +\infty$ and using \eqref{est.LinfVh} yields
  \eqref{est.espr}, and consequently $\ve w \in \cs
  L^2\loc\pqt{0,\infty;V\cap \cs H^2\hor}$.
\end{proof}

\section{Global attractor for the time-shift
  semiflow} \label{sec.attractor}

This section is devoted to the proof of Theorem
\ref{th.tsattr}. First, we recall some notions concerning semiflows
and a fundamental result that we will exploit in order to prove the
existence of the global attractor.

A bounded subset $\abset\subset W$ is called an \textdef{absorbing
  set} if for every $w\in W$, there exists $t_1=t_1(w)$ such that
$S(t)w\in \abset$ for all $t\geq t_1$. A semiflow is said
\textdef{compact} if, for every bounded set $B\subset W$ and for every
$t>0$, $S(t)B$ lies in compact subset of $W$.

We have the following result \cite{Sell:JDDE,Temam:1988}.

\begin{theorem}
  Let $S(t)$ define a compact semiflow admitting an absorbing set
  $\abset$ on a complete metric space $W$. Then $S(t)$ has a global
  attractor $\attr$ in $W$ and it coincides with the omega-limit set
  of $\abset$:
  \begin{align*}
    \attr = \bigcap_{\tau\geq 0} \overline{\bigcup_{t\geq
        \tau}S(t)\abset} \, ,
  \end{align*}
  where the closure is taken in $W$.
\end{theorem}

Hence, in order to prove Theorem \ref{th.tsattr}, it is sufficient to
verify the hypotheses of the previous result for the time-shift
semiflow, which is the content of the following propositions.

\begin{proposition}
  The mapping $\sigma$ given by $S(t)\ve w=\ve w(\cdot+t)$ is a
  semiflow.
\end{proposition}

\begin{proof}
  Clearly, $S(t)$ is a semigroup. We need to prove that the mapping
  $\uple{\tau\sep\ve w}\rightarrow S(\tau)\ve w=\ve w_{+\tau}$ is
  continuous for $\uple{\tau\sep\ve w}\in\intaa{0}{+\infty}\times\cs
  L^2\loc\pqt{0,\infty;V\hor}$. It is sufficient to show that, if
  $\tau_n$ and $\ve w^n$ are sequences such that $\tau_n\rightarrow
  \tau$ in $\intaa{0}{+\infty}$ and $\ve w^n\rightarrow \ve w$ in $\cs
  L^2\loc\pqt{0,\infty;V\hor}$, then $d(\ve w^n_{+\tau_n},\ve
  w_{+\tau})\rightarrow 0$ as $n\rightarrow +\infty$, which holds
  provided
  \begin{align*}
    \int_a^b \norm{\ve w^n_{+\tau_n}-\ve w_{+\tau}}_{V\hor}^2
    \rightarrow 0 \qquad \text{as } n\rightarrow +\infty
  \end{align*}
  for given $0\leq a<b<\infty$, where $\norm{\ve w}_{V\hor}^2 =
  \norm{\ve w}^2+\alpha^2\norm{\nabla\hor\ve w}^2$. Here and in
  the following, we omit ``$\dif s$''  in several
  integrals to keep the notation more compact.

  We can follow the steps of the proof of \cite[Lemma~7]{Sell:JDDE},
  the main difference consisting in the use of the norm
  $\norm{\cdot}_{V\hor}$ instead of $\norm{\cdot}$. Let us note that,
  since $\tau>0$, we can assume $0<\frac{1}{2}\tau \leq \tau_n \leq
  2\tau$.

  It is sufficient to prove that
  \begin{align}
 &   d\uple{\ve w^n_{+\tau_n}\sep \ve w_{+\tau_n}}\rightarrow 0 
 \label{cont1}\, ,\\
 &   d\uple{\ve w_{+\tau_n}\sep \ve w_{+\tau}}\rightarrow 0\, ,
     \label{cont2}
  \end{align}
  when $n\rightarrow +\infty$.

  As for \eqref{cont1}, we have
  \begin{align*}
    \int_a^b \norm{\ve w^n_{+\tau_n}-\ve w_{+\tau_n}}_{V\hor}^2 =
    \int_{a+\tau_n}^{b+\tau_n} \norm{\ve w^n-\ve w}_{V\hor}^2 \leq
    \int_{a+\frac{1}{2}\tau}^{b+2\tau} \norm{\ve w^n-\ve w}_{V\hor}^2
    \rightarrow 0\, ,
  \end{align*}
  as $n\rightarrow +\infty$, since $\ve w^n\rightarrow \ve w$ in $\cs
  L^2\loc\pqt{0,\infty;V\hor}$, which implies \eqref{cont1}.

  In order to show \eqref{cont2}, let us fix $\varepsilon>0$ and take
  $\ve\psi\in \cs L^2\pt{a+\frac{\tau}{2}, b+2\tau; V\hor} \cap \fs
  C^1\pt{\intcc{a+\frac{\tau}{2}}{b+2\tau};V\hor}$ such that
  \begin{align} \label{contpsi1} \int_a^b \norm{\ve
      w_{+\sigma}-\ve\psi_{+\sigma}}^2_{V\hor} \leq \varepsilon
  \end{align}
  for all $\sigma\in\intcc{\tau/2}{2\tau}$ (this is possible since
  $\fs C^1$ is dense in $\cs L^2$).

  Moreover, if $K$ denotes an upper bound for $\norm{\partial_t \ve
    \psi(s)}_{V\hor}$ for $a+\frac{1}{2}\tau\leq s\leq b+2\tau$, we
  have
  \begin{align*}
    \norm{\ve\psi(\tau_n+t)-\ve\psi(\tau+t)}_{V\hor} \leq
    \abs*{\int_{\tau_n}^\tau \norm{\partial_t\ve\psi(s+t)}_{V\hor}\dif
      s} \leq K\abs{\tau_n-\tau}\, ,
  \end{align*}
  therefore
  \begin{align} \label{contpsi2}
    \int_a^b\norm{\ve\psi_{+\tau_n}-\ve\psi_{+\tau}}_{V\hor}^2 \leq
    K^2 (b-a) \abs{\tau_n-\tau}^2 \leq \varepsilon
  \end{align}
  for $n\geq N$ sufficiently large.

  Using the triangular inequality, \eqref{contpsi1} and
  \eqref{contpsi2}, we infer that
  \begin{align*}
    & \int_a^b\norm{\ve w_{+\tau_n}-\ve w_{+\tau}}_{V\hor}^2 \leq \int_a^b \pt*{ \norm{\ve w_{+\tau_n}-\ve\psi_{+\tau_n}}_{V\hor} + \norm{\ve\psi_{+\tau_n}-\ve\psi_{+\tau}}_{V\hor} +  \norm{\ve w_{+\tau}-\ve\psi_{+\tau}}_{V\hor} }^2 \\
    & \qquad \leq 3\int_a^b \pt*{ \norm{\ve
        w_{+\tau_n}-\ve\psi_{+\tau_n}}_{V\hor}^2 +
      \norm{\ve\psi_{+\tau_n}-\ve\psi_{+\tau}}_{V\hor}^2 + \norm{\ve
        w_{+\tau}-\ve\psi_{+\tau}}_{V\hor}^2 } \leq 9\varepsilon
  \end{align*}
  for all $n\geq N$, which proves \eqref{cont2} and ends the proof.
\end{proof}

\begin{proposition}
  There exists an absorbing set $\abset\subset \mathcal W$ that is
  bounded in $\mathcal W$.
\end{proposition}
\begin{proof} We define $\abset$ as the subset of $\mathcal W$ such
  that the inequality
  \begin{align} \label{est.B} \norm{\ve
      w(t)}^2+\alpha^2\norm{\nabla\hor\ve w(t)}^2 \leq
    \dfrac{2K_1}{\nu\lambda_1} 
  \end{align}
  is satisfied for every $t\geq 0$. According to the definition of the
  metric $d$, $\abset$ is bounded in $\mathcal W$.

  We need to prove that, if $\ve w \in \mathcal W$, then $S(t)\ve w
  \in \abset$ for each $t$ sufficiently large. Actually, from
  \eqref{est.LinfVh}, there exists $t_1>0$ such that the inequality
  \eqref{est.B} holds for all $t\geq t_1$. Thus $S(t)\ve w$ belongs to
  $\abset$ for each $t\geq t_1$, and $\abset$ is an absorbing set.
\end{proof}



\begin{proposition} \label{prop.comp} The semiflow defined by $S(t)$ on the metric space
  $\mathcal W$ is compact, i.e for each bounded set $B$ in $\mathcal
  W$ and for each $t>0$, then $S(t)B$ lies in a compact subset of
  $\mathcal W$.
\end{proposition}
\begin{proof} Let $B$ be a bounded subset of $\mathcal W$.  Thanks to
  the semigroup property of $S(t)$, if $S(t)B$ is contained in a
  compact set of $\mathcal W$ for some $t>0$, then $S(t+s)B$ lies in a
  compact set of $\mathcal W$ too. Then, to prove the claim, it
  suffices to prove that $S(t)B$ lies in a compact set of $\mathcal W$
  for $0<t\leq 1$.

  Let $\{\bw^n\}$ be a bounded sequence in $\mathcal W$. Thus, there
  exists a positive constant $M_0$ such that $\int_0^1
  \big(\|\bw^n(s)ds + \alpha^2 \|\nabla\hor \bw^n(s)\|^2\big) \dif
  s\leq M_0^2$.  Recalling that $S(t)\bw^n (\tau)= \bw^n\trasl{\tau} =
  \bw^n (\tau +t)$, by the estimate in \eqref{est.LinfVh}, for $s_0\in
  \intaa{0}{t}$ and $s\geq 0$, it follows that
  \begin{align*}
    \| \bw\trasl{t}^n(s)\|^2+\alpha^2\| \nabla\hor
    \bw\trasl{t}^n(s)\|^2 &=
    \|\bw^n(s + t)\|^2+\alpha^2\| \nabla\hor \bw^n(s+ t)\|^2\\
    \leq & k_0(s_0)\nep^{-\nu\lambda_1 (t+s -s_0)} +
    \frac{K_1}{\nu\lambda_1}(1 - \nep^{-\nu\lambda_1 (t +s-s_0)}) \\
    \leq & k_0(s_0) + \frac{K_1}{\nu\lambda_1} \intertext{and,
      integrating on $\intaa{0}{1}$ in $s_0$, we reach} \|
    \bw\trasl{t}^n(s)\|^2+\alpha^2\| \nabla\hor \bw\trasl{t}^n(s)\|^2
    \leq &
    \int_0^1 \big( k_0(s_0) + \frac{K_1}{\nu\lambda_1}\big)\dif s_0 \\
    \leq & M_0^2 + \frac{K_1}{\nu\lambda_1}\, .
  \end{align*}
  Hence, we get
  \begin{equation*} 
  \| \bw\trasl{t}^n(\cdot)\|^2_{\cs
      L^{\infty}}+\alpha^2\| \nabla\hor \bw\trasl{t}^n(\cdot)\|^2_{\cs
      L^{\infty}} \leq M_0^2 + \frac{K_1}{\nu\lambda_1}\, .
  \end{equation*}
  for all $n$. Further, by
  \eqref{eq:estimate-existence}--\eqref{est.espr} we have that
  \begin{equation*}
    \begin{aligned}
      \nu \int_m^{m+1}\!\!\!\!\big(\|\nabla \bw\trasl{t}^n(s)\|^2
      +\alpha^2 \|\nabla\hor\nabla &\bw\trasl{t}^n(s)\|^2 \big) \dif s
      = \nu\int_m^{m+1}\!\!\!\!\big( \|\nabla {\ve w}^n(s +t)\|^2 +
      \alpha^2\|\nabla\hor \nabla {\ve w}^n(s + t)||^2\big) \dif s \\
      &\leq \big(\| \bw^n(t)\|^2 +\alpha^2
      \|\nabla\hor\bw^n(t)\|^2\big) + \big( 1
      +\frac{1}{\nu\lambda_1}\big)K_1\\
      &\leq \big(\| \bw^n(s_0)\|^2 +\alpha^2
      \|\nabla\hor\bw^n(s_0)\|^2\big)e^{-\nu\lambda_1 (t-s_0)} + \big( 1
      +\frac{2}{\nu\lambda_1}\big)K_1\\
      &\leq k_0(s_0) + \big( 1 +\frac{2}{\nu\lambda_1}\big)K_1
    \end{aligned}
  \end{equation*}
  for $s_0\in \intaa{0}{t}$. Therefore, integrating the above
  inequality on the interval $\intaa{0}{1}$ in $s_0$, we get
  \begin{equation*} 
    \begin{aligned}
      \int_m^{m+1}\!\!\!\!\big(\|\nabla \bw^n\trasl(s)\|^2 +\alpha^2
      \|\nabla\hor\nabla\bw^n\trasl(s)\|^2\big) \dif s &\leq \frac{M_0^2}{\nu} +
      \big(1 + \frac{2}{\nu\lambda_1}\big)\frac{K_1}{\nu}
    \end{aligned}
  \end{equation*}
  for all $n$ and for each $m=0,1,2\ldots$ The above estimates imply
  that $\pg*{S(t)\bw^n}$ is bounded in $\cs L^\infty(0, \infty;
  V\hor)\cap \cs L^2\loc\pqt*{0, \infty; V\cap \cs H^2\hor}$.
  Following the same line of reasoning as in the proof of existence in \cite{BER:2012}, one can
  also prove that $S(t) \bw^n \in \cs H^1\loc\pqt*{0, \infty;
    V^*}$.
  Therefore, using Aubin--Lions compactness theorem, up to a subsequence, $S(t)\bw^n\rightarrow \ve\gamma(t)$ strongly in $\cs L^2\loc\pqt*{0,T;V\hor}$ as $n\rightarrow +\infty$, with $\ve\gamma(t)\in \cs L^\infty(0, \infty;
  V\hor)\cap \cs L^2\loc\pqt*{0, \infty; V\cap \cs H^2\hor}$.

  The same computations also show that, up to a subsequence,
  $\ve w^n \rightarrow \ve w$ strongly in $\cs L^2\loc\pqt*{0,T;V\hor}$ as $n\rightarrow +\infty$, with $\ve w\in \cs L^\infty(0, \infty;
  V\hor)\cap \cs L^2\loc\pqt*{0, \infty; V\cap \cs H^2\hor}$. By the continuity of
  $\uple{\tau\sep\ve w}\mapsto S(t)\bw$ in
  $\intaa{0}{+\infty}\times\cs L^2\loc\pqt{0,\infty;V\hor}$ and the uniqueness
  of the limit, we actually obtain that $S(t)\ve w^n \rightarrow \ve\gamma(t)=S(t)\ve w\in \mathcal W$. This ends the proof.
%
%
\end{proof}

\section{The fully space-periodic case} \label{sec.periodic}

In this section, in order to improve the regularity of the solutions,
we consider a torus as a space domain, i.e. a domain periodic in all
directions: $D=\set{\ve x \in \ns R^3 \tc -\mpi L<x_1, x_2, x_3<\mpi
  L}$, $L>0$, with $2\mpi L$ periodicity with respect to $\ve x$. This
setting enables us to perfom some computations (more precisely, the usage of some
test functions) that are not allowed in the presence of Dirichlet
boundary conditions; in such a way we retrive some additional
information on the horizontal components of  velocity
field, $\ve w\hor$, and on the pressure, $q$.

Function spaces are defined accordingly to our periodic setting, in
particular
\begin{align*}
  & \cs L^2(D) = \set{\phi\clnf D \rightarrow \ns R \text{ measurable,
    } 2\mpi L \text{ periodic in } x,
    \, \int_D \abs{\phi}^2 \dif \ve x < +\infty}\, , \\
  & \cs L^2_0(D) = \set{\phi \in \cs L^2(D) \text{ with zero mean with
      respect to } \ve x} \, .
\end{align*}

We can still consider the horizontal filtering and thus problem
\eqref{third.1}--\eqref{third.2}, and prove the existence of a unique
regular weak solution (defined like in Section \ref{sec.setting})
essentially as done in the presence of Dirichlet boundary conditions on the top and
bottom boundary $\mathit\Gamma$. 

\subsection{Continuity of the horizontal flow}

To prove the claimed continuity of $\ve w\hor$, we need the following two lemmas.

\begin{lemma} A pressure term $q$ in \eqref{third.1} corresponding to
  a regular weak solution to \eqref{third.1}--\eqref{third.2}
  satisfies $\nabla\hor q \in \cs L^2 \pt{0,T;\cs L^2(D)}$ for each
  $T>0$.
\end{lemma}

\begin{proof} Taking the divergence operator in \eqref{third.1} and
  using \eqref{third.2}, we obtain
  \begin{align*}
    -\Delta q = \nabla\cdot \pt*{ \nabla\cdot \pt[\big]{\hf{\ve w
          \otimes \ve w}} } - \nabla\cdot\hf{\ve f},
  \end{align*}
  and, by applying the operator $A\hor$, we deduce
  \begin{align*}
    -\Delta A\hor q = \nabla\cdot \pt*{ \nabla\cdot \pt{{\ve w \otimes
          \ve w}} } - \nabla\cdot{\ve f} \, ,
  \end{align*}
  so that
  \begin{equation} \label{eq:control-for-the-pressure}
    \begin{aligned}
      \norm{A\hor q}_{\cs H^{-1}(D)} &\leq C\pt*{ \norm{\ve w\otimes
          \ve
          w}_{\cs H^{-1}(D)} + \norm{\ve f}_{\cs H^{-2}(D)}}\\
      & \leq C \pt*{ \norm{\ve w}_{\cs L^2(D)} \norm{\nabla\ve w}_{\cs
          L^2(D)} + \norm{\ve f}_{\cs H^{-2}(D)}}
    \end{aligned}
  \end{equation}
  by elliptic estimates used along with $\norm{\ve w\otimes \ve
    w}_{\cs H^{-1}(D)} \leq C \norm{\ve w}_{\cs L^2(D)}
  \norm{\nabla\ve w}_{\cs L^2(D)}$.  This last control follows from
  the H\"older and the Gagliardo--Nirenberg inequalities: taking
  $\ve\varphi\in\cs H^1(D)$, we have that
  \begin{align*}
    \abs{\uple{\ve w\otimes\ve w\sep\ve\varphi}} \leq \norm{\ve
      w}_{\cs L^3}^2 \norm{\ve\varphi}_{\cs L^3} \leq C \norm{\ve
      w}_{\cs L^2} \norm{\nabla\ve w}_{\cs L^2} \norm{\ve\varphi}_{\cs
      H^1}\, ,
  \end{align*}
  which proves the bound. Thus, from relation
  \eqref{eq:control-for-the-pressure} we get
  \begin{align*}
    \norm{A\hor q}_{\cs L^2\cs H^{-1}} \leq C \pt*{ \norm{\ve w}_{\cs
        L^\infty\cs L^2} \norm{\nabla\ve w}_{\cs L^2\cs L^2} +
      \norm{\ve f}_{\cs L^2\cs H^{-2}}}\, ,
  \end{align*}
  and hence $A\hor q \in \cs L^2 \cs H^{-1}$, so that $\nabla\hor q
  \in \cs L^2 \cs L^2$, where  $\cs L^2 \cs L^2$ denotes $\cs L^2 \pt{0,T;\cs L^2(D)}$.
\end{proof}

\begin{lemma}
  If $\ve w$ is a regular weak solution to
  \eqref{third.1}--\eqref{third.2}, then $\partial_t \nabla\hor\ve
  w\hor \in \cs L^2 \pt{0,T;V^*}$ for every $T>0$.
\end{lemma}

\begin{proof}
  From \eqref{third.1}, we have
  \begin{equation} \label{eq:h-component-time-derivative}
    \partial_t\ve w\hor = \hf{\ve f}\hor + \nu\Delta \ve w\hor
    -\nabla\hor q - \pq*{\nabla\cdot\pt{\hf{\ve w\otimes\ve w}}}\hor.
  \end{equation}
  Considering the last term in the righthand side of the above
  equation, we have
  \begin{align*}
    \norm*{\pq*{\nabla\cdot\pt{\hf{\ve w\otimes\ve w}}}\hor}_{\cs
      L^2(D)} \leq \norm*{{\nabla\cdot\pt{{\ve w\otimes\ve w}}}} \leq
    \norm*{(\ve w\cdot \nabla) \ve w} \leq C\norm{\ve w}_{\cs L^2(D)}
    \norm{\nabla\ve w}_{\cs L^2(D)}\, ,
  \end{align*}
  and therefore
  \begin{align*}
    \norm*{\pq*{\nabla\cdot\pt{\hf{\ve w\otimes\ve w}}}\hor}_{\cs
      L^2\cs L^2} \leq C\norm{\ve w}_{\cs L^\infty\cs L^2}
    \norm{\nabla\ve w}_{\cs L^2\cs L^2}\, ,
  \end{align*}
  which yields $\pq*{\nabla\cdot\pt{\hf{\ve w\otimes\ve w}}}\hor \in
  \cs L^2\cs L^2$.

  Since all the other terms on the righthand side of
  \eqref{eq:h-component-time-derivative} have the same regularity $\cs
  L^2\cs L^2$, we deduce that $\partial_t\ve w\hor\in \cs L^2 \cs
  L^2$, or $\partial_t \nabla\ve w\hor \in \cs L^2 V^*$; in
  particular, $\partial_t \nabla\hor\ve w\hor \in \cs L^2 V^*$.
\end{proof}

\smallskip

Now, we are ready to prove 
 Proposition~\ref{prop.wh}
(continuity of $\ve w\hor$).

\begin{proof}[Proposition~\ref{prop.wh}]
  Since $\ve w \in \cs L^2 (V\cap \cs H^2\hor)$, thus $\nabla\hor\ve
  w\in\cs L^2 V$, and $\partial_t \nabla\hor\ve w\hor \in \cs L^2
  V^*$, we obtain by interpolation (see
  \cite{Lio-Mag:68,SELL-YOU:2002}) that $\nabla\hor\ve w\hor \in \fs
  C^0 H$, i.e. $\ve w\hor \in \fs C^0 V\hor$. This implies the
  continuity of the map ${\ve w}_0 \mapsto \ve w\hor (t)$, with $\ve
  w_0\in V\hor$.
\end{proof}

\subsection{A higher order estimate}

We refer to the beginning of Section \ref{sec.estimates} for the
definitions of $\Lambda\hor, k_1(t)$, and set $\lambda_1=L^{-2}$
(first eigenvalue of the Laplace operator $-\Delta$ on $D$ fully
periodic, and Poincar\'e constant) and
\begin{equation} \label{eq:utility-quantities}
  \begin{aligned}
    & k_2(t) =\|\nabla \bw(t)\|^2+ \alpha^2\|\nabla \nabla\hor \bw(t)\|^2\, , \\
    & K_2 = \frac{3}{\nu}\min \pg*{\frac{\norm{\Lambda\hor^{-1}\ve
          f}^2}{\alpha^2} \sep
      \norm{\ve f}^2} \, ,\\
    & k_3(t) = K_2 + \frac{C k_1(t)^3}{\alpha^8\nu^3}\Big(
    \frac{1}{\alpha^4
      \lambda_1^3}+ \frac{k_1(t)^2}{\nu^4}\Big)\, ,\\
    & k_4(t) = k_2(t) + \frac{k_3(t)}{\nu\lambda_1}\, .
  \end{aligned}
\end{equation}

We state and prove the following result.

\begin{theorem} \label{thm:estimate-H2_h} Let $\bw$ be a regular weak
  solution to \eqref{third.1}--\eqref{third.2} in the time interval
  $\intca{0}{\infty}$; then we have that
  \begin{equation} \label{eq:final-estimate-H2h} \|\nabla
    \bw(t+r)\|^2+\alpha^2\| \nabla\nabla\hor \bw(t+r)\|^2\leq
    k_2(t)\nep^{-\nu\lambda_1 (t+r)} + \frac{k_3(t)}{\nu\lambda_1}(1 -
    \nep^{-\nu\lambda_1(t+r)}) \leq k_4(t)
  \end{equation}
  for each $t, r>0$.
\end{theorem}
\begin{proof} In the space periodic setting we can use $-\Delta
  A\hor\bw$ as test function for Equation \eqref{third.1} (now
  formally, but the procedure actually goes through the Galerkin
  approximation).

  First, observe that
  \begin{align*}
    \abs{\uple{{\ve f}\sep \Delta\ve w}} &\leq \begin{cases}
      &\norm{\Lambda\hor^{-1}\ve f} \, \norm{\Lambda\hor\Delta\ve w} \\*[1em]
      &\norm{\ve f} \, \norm{\Delta\ve w}
    \end{cases}
    \leq \begin{cases} &\dfrac{\norm{\Lambda\hor^{-1}\ve
          f}^2}{\nu\alpha^2} +
      \dfrac{\nu\alpha^2\norm{\Delta\nabla\hor\ve w}^2}{4} \\*[1em]
      &\dfrac{3\norm{\ve f}^2}{2\nu} + \dfrac{\nu\norm{\Delta \ve
          w}^2}{6}
    \end{cases} \\
    & \leq \dfrac{1}{2}K_2+
    \frac{\nu\alpha^2}{4}\norm{\Delta\nabla\hor\ve w}^2 +
    \frac{\nu}{6}\norm{\Delta\ve w}^2\, .
  \end{align*}
  Thus, testing the equation \eqref{third.1} against $-\Delta A\hor
  \bw$, we get
  \begin{equation} \label{eq:second-estimate-w}
    \begin{aligned}
      \frac{1}{2} \deriv{}{t}\big( \|\nabla \bw\|^2 +& \alpha^2\|
      \nabla \nabla\hor \bw\|^2\big) +
      \nu \big( \|\Delta \bw\|^2 + \alpha^2 \| \Delta \nabla\hor  \bw\|^2\big)\\
      &\leq \big|\big(\bbf, \overline{A\hor \Delta \bw}^h\big)\big| +
      \big|\big(\nabla \cdot (\bw\otimes \bw),
      \overline{A\hor \Delta \bw}^h\big)\big| \\
      &\leq \dfrac{1}{2}K_2+
      \frac{\nu\alpha^2}{4}\norm{\Delta\nabla\hor\ve w}^2 +
      \frac{\nu}{6}\norm{\Delta\ve w}^2 + \big|\big( (\bw\cdot \nabla)
      \bw, \Delta \bw \big)\big| \, .
    \end{aligned}
  \end{equation}

  Take into account the nonlinear term in the righthand side of the
  above estimate. We have that
  \begin{equation} \label{eq:terms-A-B}
    \begin{aligned}
      \big|\big( (\bw\cdot \nabla) \bw, \Delta \bw \big)\big| &=
      \big|\big(\nabla \big[(\bw\cdot \nabla) \bw\big], \nabla \bw
      \big)\big|
      =\Big|\int_{D} \nabla  \big[(\bw\cdot \nabla) \bw\big] : \nabla \bw dx\Big| \\
      &\leq \big|\big(\nabla\hor \big[(\bw\cdot \nabla) \bw\big],
      \nabla\hor \bw \big)\big| + \big|\big(\partial\vrt
      \big[(\bw\cdot \nabla) \bw\big],
      \D\vrt \bw \big)\big| \\
      &= \big|\big( (\nabla\hor \bw\cdot \nabla) \bw, \nabla\hor \bw
      \big)\big| + \big|\big((\D\vrt\bw\cdot \nabla
      ) \bw, \D\vrt\bw \big)\big|\\
      & =: A + B\, ,
    \end{aligned}
  \end{equation}
  where in the last step we used the following relations
  \begin{equation*}
    \big((\bw\cdot \nabla) \nabla\hor \bw, \nabla\hor \bw \big)=0 \qquad
    \textrm{ and } \qquad
    \big((\bw \cdot \nabla) \D\vrt\bw,
    \D\vrt\bw \big) =0\, .
  \end{equation*}
  Therefore, exploiting \cite[\S 4, (4.12)]{BER:2012} we obtain
  \begin{equation*}
    \begin{aligned}
      A &= \Big|\Big( \big([\nabla\hor \bw]\hor\cdot \nabla\hor\big)
      \bw +
      [\nabla\hor \bw]_3 \D\vrt\bw, \nabla\hor \bw \Big)\Big|\\
      &\leq \Big|\Big( \big(\nabla\hor \bw\hor\cdot \nabla\hor\big)
      \bw, \nabla\hor \bw \Big)\Big| + \Big|\Big(\nabla\hor
      w_3\D\vrt\bw,
      \nabla\hor \bw \Big)\Big| \\
      &=: A_1 + A_2\, ,
    \end{aligned}
  \end{equation*}
  where {$[ \, \cdot\, ]\hor$ and $[\, \cdot\, ]_3$} indicate,
  respectively, the horizontal and vertical components of the
  considered vector fields.

  Next, we estimate the nonlinear terms $A_1$, $A_2$ and $B$ defined
  above.  First, we estimate $A_1$ and we get
  \begin{equation} \label{eq:term-A1}
    \begin{aligned}
      A_1 &\leq \int_{D} |\nabla\hor \bw\hor||\nabla\hor \bw|^2 \\
      &\leq \|\nabla\hor \bw \|^2_{\cs L^4} \|\nabla\hor \bw\hor\|\\
      &\leq \|\nabla \nabla\hor \bw\|^{3/2}\| \nabla\hor \bw\|^{1/2}
      \|\nabla\hor \bw\hor\|\\
      &\leq \frac{1}{4}\nu\alpha^2\lambda_1 \|\nabla \nabla\hor
      \bw\|^2 + \frac{C}{\nu^3\alpha^6\lambda_1^3}
      \| \nabla\hor \bw\|^2\| \nabla\hor \bw\hor\|^4\\
      &\leq \frac{1}{4}\nu\alpha^2 \|\Delta \nabla\hor \bw\|^2 +
      \frac{C}{\nu^3\alpha^6\lambda_1^3}\| \nabla\hor \bw\|^6\, .
    \end{aligned}
  \end{equation}
  For the term $A_2$ we have the following control
  \begin{equation} \label{eq:term-A2}
    \begin{aligned}
      A_2&\leq \int_{D} |\nabla\hor \bw| |\D_3 \bw| |\nabla\hor w_3|\\
      & \leq \int_{D} |\nabla \bw|^2 |\nabla\hor w_3|\\
      &\leq \|\nabla \bw\|^2_{\cs L^4} \|\nabla\hor w_3\|\\
      &\leq C \|\Delta  \bw\|^{7/4}\| \bw\|^{1/4} \| \nabla\hor w_3\| \\
      &\leq \frac{1}{6}\nu\|\Delta \bw\|^2 + \frac{C}{\nu^7} \|
      \bw\|^2 \|\nabla\hor w_3\|^8\, ,
    \end{aligned}
  \end{equation}
  where we used the Gagliardo--Nirenberg inequality $\| D^j \bw\|_{\cs
    L^p}\leq C \|D^m \bw\|^{\alpha}\|_{\cs L^r}\|\bw\|_{\cs
    L^q}^{1-\alpha}$, with $j=1$, $p=4$, $m=r=q=2$ and $\alpha=7/8$,
  and the Young's inequality.

  Let us now consider the term $B$ in \eqref{eq:terms-A-B}. We have
  that
  \begin{equation*}
    \begin{aligned}
      B &= \Big|\Big( \big([\D_3 \bw]\hor\cdot \nabla\hor \big) \bw +
      [\D_3 \bw]_3\D_3 \bw, \D_3 \bw \Big) \Big| \\
      &\leq \Big|\Big( \big(\D_3 \bw\hor\cdot \nabla\hor \big) \bw,
      \D_3 \bw \Big) \Big| + \Big|\Big( \D_3 w_3\D_3 \bw, \D_3 \bw
      \Big) \Big| \, .
    \end{aligned}
  \end{equation*}
  Now, arguing as in the case of $A_2$, the term $B$ can be controlled
  as follows.
  \begin{equation} \label{eq:term-B}
    \begin{aligned}
      B & \leq \int_{D} |\D_3 \bw\hor| |\D_3 \bw||\nabla\hor \bw|
      +  \int_{D} |\D_3 w_3| |\D_3 \bw|^2 \\
      & \leq \|\D_3 \bw\|^2_{\cs L^4}\|\nabla\hor \bw\| + \|\D_3
      \bw\|^2_{\cs L^4} \|\D_3 w_3\| \\
      &\leq \|\nabla \bw\|^2_{\cs L^4}\big( \|\nabla\hor \bw\| + \|\D_3 w_3\|\big) \\
      &\leq C \|\Delta \bw\|^{7/4}\| \bw\|^{1/4}
      \big( \|\nabla\hor \bw\| + \|\D_3 w_3\|\big) \\
      &\leq \frac{1}{6}\nu\|\Delta \bw\|^2 + \frac{C}{\nu^7} \|
      \bw\|^2 \big( \|\nabla\hor \bw\| + \|\D_3 w_3\|\big)^8\, .
    \end{aligned}
  \end{equation}

  Collecting \eqref{eq:second-estimate-w}, \eqref{eq:term-A1},
  \eqref{eq:term-A2}, and \eqref{eq:term-B}, we obtain
  \begin{equation} \label{est.H2h}
    \begin{split}
      & \frac{1}{2} \deriv{}{t}\big( \|\nabla \bw \|^2 +
      \alpha^2\|\nabla\hor \nabla \bw\|^2\big) +
      \nu \big( \|\Delta \bw\|^2 + \alpha^2 \| \nabla\hor  \Delta \bw\|^2\big)\\
      & \qquad \leq \frac{1}{2}\nu \big( \|\Delta \bw\|^2 +
      \alpha^2\|\Delta \nabla\hor \bw\|^2 \big) +
      \frac{C}{\nu^3\alpha^6 \lambda_1^3} \| \nabla\hor \bw\|^6
      +  \frac{C}{\nu^7} \| \bw\|^2 \|\nabla\hor w_3\|^8\\
      & \qquad \qquad + \frac{C}{\nu^7} \| \bw\|^2 \big( \|\nabla\hor
      \bw\| +
      \|\D_3 w_3\|\big)^8 +     \frac{K_2}{2} \\
      & \qquad \leq \frac{1}{2}\nu \big( \|\Delta \bw\|^2 +
      \alpha^2\|\Delta \nabla\hor \bw\|^2 \big) + \frac{K_2}{2} +
      \frac{C k_1(t)^3}{\alpha^8\nu^3}\Big( \frac{1}{\alpha^4
        \lambda_1^3}+ \frac{k_1(t)^2}{\nu^4}\Big) \, ,
    \end{split}
  \end{equation}
  where in the last step we have used \eqref{est.LinfVh} and $\D_3 w_3
  = -\nabla\hor \cdot \bw\hor$ (which imply $\bw \in \cs
  L^\infty\pt*{0,\infty;V\hor}$ as well as $w_3 \in \cs
  L^\infty\pt*{0,\infty; \cs H^1}$). Thanks to the Poincar\'e's
  inequality, we easily reach
  \begin{equation} \label{est.y} y'(t)+ \nu \lambda_1 y(t) \leq
    k_3(t)\, ,
  \end{equation}
  having set
  \begin{equation} \label{def.y} y(t)=\|\nabla \bw(t)\|^2 + \alpha^2
    \|\nabla \nabla\hor \bw(t)\|^2\, .
  \end{equation}
  Hence, by applying Gronwall's inequality in the interval
  $\intcc{t}{t+r}$, we infer that
  \begin{gather*}
    \|\nabla \bw(t+r)\|^2+\alpha^2\| \nabla\nabla\hor \bw(t+r)\|^2\leq
    k_2(t)\nep^{-\nu\lambda_1(t+r)} + \frac{k_3(t)}{\nu\lambda_1}(1 -
    \nep^{-\nu\lambda_1(t+r)})\, ,
  \end{gather*}
  and finally
  \begin{gather*}
    \|\nabla \bw(t+r)\|^2+\alpha^2 \|\nabla \nabla\hor \bw(t+r)\|^2
    \leq k_4(t)\, ,
  \end{gather*}
  by definition of $k_4(t)$. Thus, the conclusion follows.
\end{proof}

\medskip

\textbf{Acknowledgement.} The authors are members of the Gruppo Nazionale per
l'Analisi Mate\-ma\-tica, la Probabilit\`a e le loro Applicazioni (GNAMPA)
of the Istituto Nazionale di Alta Matematica (INdAM).

Luca Bisconti, Dip. di Matematica e Informatica ``U. Dini'', Universit\`a degli Studi di Firenze, Via S. Marta 3, I-50139, Firenze, Italia.

e-mail: luca.bisconti@unifi.it

\smallskip

Davide Catania, DICATAM, Sezione Matematica, Universit\`a degli Studi di Brescia, Via Valotti 9, I-25133, Brescia, Italia.

e-mail: davide.catania@unibs.it

\end{document}